\documentclass{amsart}
\usepackage{graphicx,amsfonts,amssymb,amsmath,amsthm,xcolor}
\usepackage[pdftex]{hyperref} 
\usepackage{cite}  
\usepackage{enumitem} 

\theoremstyle{plain} 
\newtheorem{theorem}    {Theorem}[section]
\newtheorem{lemma}      [theorem]{Lemma}

\theoremstyle{definition}

\theoremstyle{remark}
\newtheorem{remark}              [theorem]{Remark}

\newtheorem{question}   [theorem]{Question}

\numberwithin{equation}{section}

\def\A{\mathbb A}
\def\C{\mathbb C}
\def\F{\mathbb F}

\def\Q{\mathbb Q}

\begin{document}

\title{A conjectural refinement of strong multiplicity one for {\rm GL}($n$)}

\author[N. Walji]{Nahid Walji}
\address{Department of Mathematics, University of British Columbia, Vancouver, B.C., V6T 1Z2, Canada}
\email{nwalji@math.ubc.ca}

\maketitle
\begin{abstract} 
Given a pair of distinct unitary cuspidal automorphic representations for {\rm GL}($n$) over a number field, let $S$ denote the set of finite places at which the automorphic representations are unramified and their associated Hecke eigenvalues differ. In this note, we demonstrate how conjectures on the automorphy and possible cuspidality of adjoint lifts and Rankin--Selberg products imply lower bounds on the size of $S$. We also obtain further results for ${\rm GL}(3)$.\\
\end{abstract}

\section{Introduction}
Let $\pi$ and $\pi'$ be two distinct unitary cuspidal automorphic representations for ${\rm GL}(n)/F$, where $F$ is a number field. Define $X=X(\pi,\pi')$ to be the set containing exactly the archimedean places and the finite places at which either $\pi$ or $\pi'$ is ramified.
Associated to each place $v \not \in X$, is the multiset of $n$ Satake parameters for $\pi$. The Hecke eigenvalue associated to $\pi$ at $v$, denoted $a_v(\pi)$,  is then defined to be the sum of these Satake parameters.
We ask
\begin{question}\label{Qu1}
What can be said about the size of the set
\begin{align*}
S (\pi,\pi'):=\{v \not \in X \mid a_v(\pi) \neq a_v(\pi') \}\ ?
\end{align*}
\end{question}

An answer to this was provided by Jacquet and Shalika~\cite{JS81}, who proved that if $\pi \not \simeq \pi'$, then $S=S(\pi,\pi')$ is infinite.
One expects a stronger bound on the size of $S$ to hold;
 a conjecture of Ramakrishnan~\cite{Ra94b} states that given unitary cuspidal automorphic representations $\pi$ and $\pi'$ for ${\rm GL}(n)/F$, if $\pi \not \simeq \pi'$, then $\underline{\delta} (\pi,\pi') \geq 1/2n ^2$ (where $\underline{\delta}(\pi,\pi')$ represents the lower Dirichlet density of the set $S(\pi,\pi')$).
This bound is known to be a consequence of the Ramanujan conjecture. It is unconditional in the case of $n = 2$~\cite{Ra94a}, but not known for any larger $n$. For every $n$, the conjectured bound would be sharp, given a family of examples of Serre~\cite{Se81}.
In terms of progress towards this bound, a result of Rajan~\cite{Ra03} states that if $\pi \not \simeq \pi'$ then $S$ has the property that the sum
\begin{align*}
\sum_{v \in S}{\rm N}v ^{- 2/ (n ^2 + 1)},
\end{align*}
does not converge. 

In {\rm GL}(2), more is known. In addition to the conjecture of Ramakrishnan, a stronger (and sharp) bound of 1/4 (on the lower Dirichlet density) holds~\cite{Wa14} under the condition that neither of the automorphic representations correspond to an Artin representation of dihedral type.

For {\rm GL}(3), unconditional positive bounds are not known; however, based on a study of suitable finite group representations, one has the prediction that for $\pi$ and $\pi'$ that are not automorphically induced from characters, a bound of 2/7 should hold~\cite{MW17}.

For {\rm GL}(n) where $n$ is greater than 2, there are much fewer cases of functoriality known, making unconditional progress difficult. This is in contrast to the situation for {\rm GL}(2). For example, in 1978, the work of Gelbart--Jacquet~\cite{GJ78} proved that the adjoint lift from {\rm GL}(2) to {\rm GL}(3) (which is twist-equivalent to the symmetric square lift) is automorphic.
A cuspidality criterion was also obtained: the adjoint lift of $\pi$ is cuspidal if and only if $\pi$ is not automorphically induced from a Hecke character.
Ramakrishnan~\cite{Ra00} later showed that if the adjoint lifts of two cuspidal automorphic representations $\pi$,$\pi'$ for {\rm GL}(2) are equal, then $\pi$ and $\pi'$ must be twist-equivalent.
The Langlands functoriality conjectures include the prediction that the adjoint lift from ${\rm GL}(n)$ to ${\rm GL}(n^2-1)$ should be automorphic for any $n > 2$, but this is far from known in general. 

One can ask about the implications of assuming the automorphy, and cuspidality, of the adjoint lift with respect to obtaining bounds in response to question~\ref{Qu1}.
\begin{theorem}\label{t1}
Let $\pi$ and $\pi'$ be distinct unitary cuspidal automorphic representations for ${\rm GL}(n)/F$, and assume that the adjoint lifts of $\pi$ and $\pi'$ are automorphic. \\ 
(a) If both adjoint lifts are cuspidal, then $$\underline{\delta} (\pi,\pi') \geq \frac{1}{8}.$$ 
(b) In addition to the cuspidality assumption, if furthermore we assume that the adjoint lifts are distinct, then
\begin{align*}
\underline{\delta}(\pi,\pi') \geq \frac{1}{3 + 2 \sqrt{2}}.
\end{align*}
\end{theorem}

\begin{remark}
Note that  $1/(3 + 2 \sqrt{2})= 0.1715...$ . We also remark that these bounds are uniform in $n$. In Ramakrishnan's conjecture the bound tends to zero as $n \rightarrow \infty$. The contrast with the uniform bound obtained here might be seen as a reflection of the strength of our assumptions in Theorem~\ref{t1}.
\end{remark}

These bounds are not expected to be sharp. Certainly, in the case of {\rm GL}(2) one knows that the optimal bounds are $1/4$ and $2/5$, respectively~\cite{Wa14}. Because we are only assuming one instance of functoriality, namely the automorphy of the adjoint lift, this imposes restrictions on the extent to which we can implement the method that we use.
If we also include assumptions about the automorphy and cuspidality of certain Rankin--Selberg products of $\pi, \pi'$ and their duals  $\widetilde{\pi}, \widetilde{\pi'}$ (respectively), then we can obtain stronger bounds:

\begin{theorem}\label{t2}
Let $\pi$ and $\pi'$ be distinct unitary cuspidal automorphic representations for ${\rm GL}(n)/F$ whose adjoint lifts and Rankin--Selberg products $\pi \boxtimes \pi'$ and $\pi \boxtimes \widetilde{\pi'}$, are all automorphic and furthermore cuspidal. 
If the adjoint lifts of $\pi$ and $\pi'$ are distinct, then 
\begin{align*}
\underline{\delta}(\pi,\pi') \geq \frac{2}{5}.
\end{align*}
\end{theorem}

\begin{remark}
This bound is already known unconditionally when $n = 2$~\cite{Wa14}, in which case it is sharp. We do not expect the bound to be sharp for all $n$ (see, for example,~\cite{MW17}).
\end{remark}

In the case of ${\rm GL}(3)$, we obtain further bounds without needing to appeal to a cuspidality assumption. First, we define some terminology: Let $\Pi$ be a cuspidal automorphic representation for ${\rm GL}(3)$.
We say that $\Pi$ is \textit{associated} to a cuspidal automorphic representation $\pi$ for ${\rm GL}(2)$ if $\Pi$ is twist-equivalent to ${\rm Ad}(\pi)$ (this will be the case if $\Pi$ is essentially self-dual~\cite{Ra14}).
Additionally, if there exists an isobaric automorphic representation $\tau$ such that for some finite set $S$ of places we have $L^S(s, \tau)= L^S(s, \Pi, {\rm Ad})$, then we will say that the adjoint lift of $\Pi$ is \textit{weakly automorphic}.

\begin{theorem}\label{gl3}
Let $\Pi$ and $\Pi'$ be unitary cuspidal automorphic representations for ${\rm GL}(3)$ over a number field $F$. \\
(a) Assume that the adjoint lifts of $\Pi$ and $\Pi'$ are weakly automorphic, then 
\begin{align*}
\underline{\delta}(\Pi,\Pi') \geq 1/ 28
\end{align*}
(b) If $\Pi$ and $\Pi'$ are essentially self-dual, then
\begin{align*}
\underline{\delta}(\Pi,\Pi') \geq 2/(17+3\sqrt{21})  
\end{align*}
(c) If $\Pi$ and $\Pi'$ are essentially self-dual and not associated to a cuspidal automorphic representation for ${\rm GL}(2)$ of solvable polyhedral type, then
\begin{align*}
\underline{\delta}(\Pi,\Pi') \geq 1/12.
\end{align*}
\end{theorem}

\begin{remark}
A cuspidal automorphic representation that is of solvable polyhedral type means that it is either of dihedral, tetrahedral, or octahedral type.
Also, note that parts (b) and (c) do not rely on any conjectures about functorial lifts.
\end{remark}

For completeness, we also mention another question in this context. Consider two cuspidal automorphic representations $\pi$ and $\pi'$ for {\rm GL}(n) over a number field $F$, both with analytic conductor less than some $Q$. What bound  $C=C(n,F,Q)$ can one establish such that if $\pi_{\mathfrak{p}} \simeq \pi'_\mathfrak{p}$ for all ${\rm N} \mathfrak{p} \leq C$, then it implies that $\pi = \pi'$?

Answers to this question were established by Moreno~\cite{Mo85}, and later Brumley~\cite{Br06}; further improvements were subsequently obtained by Wang~\cite{Wa08} and then Liu--Wang~\cite{LW09}. In this last paper, the authors show that it suffices to have, for any $\epsilon > 0$, a bound $C = cQ^{2n+\epsilon}$ with a suitable constant $c = c(n,F,\epsilon)$.

We also note related work of Murty--Rajan~\cite{MR96}. Let $\pi$ and $\pi'$ be cuspidal automorphic representations for {\rm GL}(2)/$\Q$. Assume that the $L$-functions of the form $L(s,{\rm Sym}^a(\pi)\times {\rm Sym}^b (\pi'))$, where $a=b$ or $a=b+2$, are entire, have suitable functional equations, and satisfy GRH. Then, for any $\epsilon > 0$,
\begin{align*}
\#\{p \leq x \mid a_p(\pi) = a_p(\pi') \} \ll x ^{5/6+ \epsilon}.
\end{align*}

This paper is organised as follows. In Section~\ref{backg}, we introduce some notation for the paper and cover the relevant background. In Section~\ref{exs}, we provide some context for the condition in the first two theorems of the cuspidality of the adjoint lift; various cases are described (some of which rely on the strong Artin conjecture) where the isobaric decomposition of the adjoint lift is discussed. In Section~\ref{pf1}, we cover the proofs of Theorems~\ref{t1} and~\ref{t2}. Lastly, in Section~\ref{pf2} we prove Theorem~\ref{gl3}.

\section{Background and notation}\label{backg}

Throughout this paper we will write $\ell(s)$ to denote $\log (1/(s-1))$. The \textit{lower Dirichlet density} for a set $S$ of places is defined to be 
\begin{align*}
\underline{\delta}(S) := \liminf_{s \rightarrow 1^+} \left(\frac{\sum_{v \in S}{\rm N}v^{-s} }{\ell(s)}\right) 
\end{align*}
and similarly for the \textit{upper Dirichlet density} $\overline{\delta}(S)$ if we replace the limit infimum by the limit supremum. If the upper and lower Dirichlet density are equal then the set has a Dirichlet density, denoted by $\delta (S)$.

\subsection{$L$-functions} For  an  automorphic representation $\pi$ for {\rm GL}(n) over a number field $F$, define $X=X(\pi)$ to be the set of all archimedean places as well as the places at which the automorphic representation is ramified. The incomplete $L$-function of $\pi$ with respect to $X$ is 
\begin{align*}
L^X(s,\pi) = \prod_{v \not \in X} {\rm det}\left(I_n - A_v(\pi) {\rm N}v^{-s} \right)^{-1} ,
\end{align*}
where $A_v(\pi)$ is a conjugacy class in ${\rm GL}_n(\C)$ that can be represented by a diagonal matrix that we denote as ${\rm diag}(\alpha_1,\alpha_2, \dots, \alpha_n)$. This $L$-function converges absolutely for ${\rm Re}(s)>1$.

Given a pair of unitary automorphic representations $\pi$ and $\pi'$ for {\rm GL}(n) and {\rm GL}(m), respectively, over some number field $F$, we let $X=X (\pi, \pi')$ be the set of archimedean places and places where either $\pi$ or $\pi'$ is ramified. Then the incomplete Rankin--Selberg $L$-function with respect to $X$ is
\begin{align*}
L^X(s, \pi \times \pi') = \prod_{v \not \in X} {\rm det}\left(I_{nm} - (A_v(\pi)\otimes A_v(\pi')) {\rm N}v^{-s} \right)^{-1},
\end{align*}
which converges absolutely for ${\rm Re}(s)>1$. It has a simple pole at $s=1$ if and only if $\pi'$ is dual to $\pi$~\cite{JS81}; otherwise it is invertible there~\cite{Sh81}.

The automorphy of the Rankin--Selberg product is known in the case of ${\rm GL}(2) \times {\rm GL}(2)$ due to Ramakrishnan~\cite{Ra00}, who also provides a cuspidality criterion. Though we do not use explicit cuspidality criteria in this paper and instead directly assume that certain Rankin--Selberg products are automorphic and cuspidal, for context we outline the criteria for the ${\rm GL}(2) \times {\rm GL}(2)$ case:
If $\pi$ and $\pi'$ are cuspidal and neither is automorphically induced from {\rm GL}(1), then $\pi \boxtimes \pi'$ is cuspidal if and only if $\pi$ and $\pi'$ are not twist-equivalent. If $\pi$ is automorphically induced from a Hecke character $\chi$ for ${\rm GL}(1)/K$, where $K$ is a quadratic extension of $F$, then $\pi \boxtimes \pi'$ is cuspidal if and only if the base change $\pi'_K$ is cuspidal and distinct from $\pi'_K \otimes (\chi \circ \tau)\chi^{-1}$, where $\tau$ is the non-trivial element of ${\rm Gal}(K / F)$.

Automorphy is also known in the case ${\rm GL}(2) \times {\rm GL}(3)$, due to Kim--Shahidi~\cite{KS00}, and with a cuspidality criterion proved in~\cite{RW04}. Given $\pi$ for ${\rm GL}(2) /F$ and $\pi'$ for ${\rm GL}(3) /F$, $\pi \boxtimes \pi'$ is cuspidal unless:
$\pi$ is not dihedral and its adjoint lift is twist-equivalent to $\pi'$,
or $\pi$ is dihedral and $\pi'$ is the automorphic induction of a Hecke character over a non-Galois cubic extension $L$ and $\pi'_L$ is Eisensteinian.

Given automorphic representations $\pi_1, \dots, \pi_4$ for ${\rm GL}(n_1), \dots, {\rm GL}(n_4)$ (respectively) over $F$, we can define, in a suitable right-half plane, the incomplete quadruple product $L$-function
\begin{align*}
L^X&(s, \pi_1 \times \pi_2 \times \pi_3 \times \pi_4) = \\
&\prod_{v \not \in X} {\rm det}\left(I_{n_1n_2n_3n_4} - (A_v(\pi_1)\otimes A_v(\pi_2) \otimes A_v(\pi_3) \otimes A_v(\pi_4) ) {\rm N}v^{-s} \right)^{-1},
\end{align*}
where $X= X(\pi_1, \pi_2, \pi_3, \pi_4)$ is the set of archimedean places and places at which at least one of the automorphic representations is ramified. 

We describe the construction of an adjoint $L$-function via the use of a Rankin--Selberg $L$-function and Dedekind zeta function
\begin{align*}
L(s, \pi, {\rm Ad}) := \frac{L(s, \pi \times \widetilde{\pi}) }{  \zeta_F (s)},
\end{align*}
and the adjoint $L$-function is meromorphic since $L(s, \pi \times \widetilde{\pi})$ and $\zeta_F (s)$ are meromorphic.
We now write, for $X= X (\pi)$,
\begin{align*}
L^X(s,\pi, {\rm Ad}) = \prod_{v \not \in X} {\rm det}\left(I_{n ^2 -1} - A_v(\pi, {\rm Ad}) {\rm N}v^{-s} \right)^{-1} ,
\end{align*}
where $A_v(\pi, {\rm Ad}) \in {\rm GL}_{n ^2 -1}(\C) $ can be represented by a diagonal matrix with eigenvalues $\alpha_i / \alpha_j$, for all $i,j$ with $i \neq j$, and $1$ ($n-1$ times).

Let $\pi$ be a cuspidal automorphic representation for ${\rm GL}_n(\A_F)$. If $n = 2$, then its adjoint lift ${\rm Ad}\pi$ is known to be automorphic, and furthermore, ${\rm Ad}\pi$ is cuspidal if and only if $\pi$ is not an automorphic induction of a Hecke character~\cite{GJ78}. For $n > 2$, the conjecture that the adjoint lift of $\pi$ is automorphic is wide open.

\subsection{Isobaric automorphic representations} Let $\pi_1,\pi_2, \dots , \pi_k$ be cuspidal automorphic representations for ${\rm GL}(n_1)$, ${\rm GL}(n_2)$, $\dots {\rm GL}(n_k)$, respectively, all over the same number field $F$. Then there exists an automorphic representation $\Pi$ for ${\rm GL}(n_1+\dots +n_k) $ such that, for any cuspidal automorphic representation $\pi'$ for some ${\rm GL}(m)/F$,
\begin{align*}
L(s, \Pi \times \pi') = \prod_{j = 1 }^{ k}L(s, \pi_j \times \pi').
\end{align*}
Such a $\Pi$ is called an \textit{isobaric automorphic representation} and we write $\Pi$ as $\boxplus_{j = 1 }^{ k} \pi_j$. See \cite{La79} and~\cite{JS81} for further details.
Note that, at every place $v$ where $\pi_1, \dots, \pi_k$ are unramified, we have $a_v(\Pi) = a_v(\pi_1) + \dots + a_v(\pi_k)$.

We also note the following: Let $\tau_1 \boxplus \dots \boxplus \tau_k $ and $ \sigma_1 \boxplus \dots \sigma_m$ be isobaric automorphic representations. If they are isomorphic, then $k = m$ and there is a permutation $P$ on $\{1, \dots, k\}$ such that $\tau_j = \sigma_{P(j)}$.

\subsection{Weak automorphy} 
We briefly recall the definition from the introduction. Given a cuspidal automorphic representation $\Pi$ for {\rm GL}(3), we will say that its adjoint lift is \textit{weakly automorphic} if there exists an isobaric automorphic representation $\tau$ (for ${\rm GL}(n^2-1)$) such that $L^S(s, \Pi, {\rm Ad}) = L^S(s, \tau)$ for a finite set $S$.

Note that, in general, equality of incomplete $L$-functions does not imply  automorphy. For example, Cogdell and Piatetski-Shapiro~\cite{CPS99} have shown that there exist infinitely many irreducible admissible representations for ${\rm GL}_4(\A)$ whose $L$-functions coincide with that of some fixed automorphic representation $\Pi$, but they are not automorphic.

\section{Examples}\label{exs}

To provide some context for the cuspidality condition in Theorems~\ref{t1} and~\ref{t2}, we discuss some examples of adjoint lifts for cuspidal automorphic representations for ${\rm GL}(n)$. The situation is better known when $n = 2$, or for certain special cases (such as when $n = 3$ and the representation is essentially self-dual). In the second subsection, we assume the strong Artin conjecture to obtain other types of examples.

\subsection{Examples arising from functoriality theorems} Langlands' principle of functoriality says that the adjoint lift of a cuspidal automorphic representation $\pi$ for ${\rm GL}(n)$ over a number field $F$ is expected to be an automorphic representation for ${\rm GL}(n^2-1)/F$. One might expect it to be cuspidal provided that $\pi$ does not arise from a transfer from another group to ${\rm GL}(n)$. For example, let $n = 2$ and assume that $\pi$ is monomial, that is to say, it is the automorphic induction of a Hecke character over a number field $K$ that is a quadratic extension of $F$. Then by Gelbart--Jacquet~\cite{GJ78} we know that the adjoint lift is not cuspidal but rather has an isobaric decomposition, into either a pair of {\rm GL}(1) and a {\rm GL}(2) cuspidal automorphic representations, or into three {\rm GL}(1) cuspidal automorphic representations. 
If, on the other hand, $\pi$ is not monomial, then ${\rm Ad}\pi$ is cuspidal.

Outside of ${\rm GL}(2)$, much less is known. If $n = 3$, then the situation is understood for essentially self-dual cuspidal automorphic representations: Such a representation $\Pi$ is known (for example, see~\cite{Ra14}) to be twist-equivalent to the adjoint lift of a cuspidal automorphic representation from {\rm GL}(2) to {\rm GL}(3). Let $\omega_\pi$ denote the central character of $\pi$ and $X=X (\Pi)$ be the (finite) set of places which are either archimedean or at which $\Pi$ is ramified. Then it is well known (via Clebsch--Gordon decomposition) that
\begin{align*}
L^X(s, \Pi, {\rm Ad})= L^X(s, {\rm Sym}^4 \pi \otimes \omega_\pi ^{-2})L^X(s, {\rm Ad}\pi),
\end{align*}
which means that the adjoint lift of $\Pi$ cannot be cuspidal.

\subsection{Examples arising from the strong Artin conjecture}
Given the limited number of unconditional examples, we turn to examples based on the strong Artin conjecture. Given a representation $\rho: {\rm Gal}(\overline{F}/F) \rightarrow {\rm GL}_n(\C)$, the strong Artin conjecture implies that there exists an automorphic representation $\Pi$ for ${\rm GL}(n)/F$ such that the $L$-functions of these two objects are equal 
\begin{align*}
L (s, \rho) = L(s,\Pi)
\end{align*}
and that $\rho$ is irreducible if and only if $\Pi$ is cuspidal.

Given a Galois representation $\rho: {\rm Gal}(\overline{F}/F) \rightarrow {\rm GL}_3(\C)$, we can classify it according to the group structure of its projective image in ${\rm PGL}_3(\C)$, which has been done by Blichfeldt~\cite{Bl17}. We now define two such groups. Fix a non-trivial ninth root of unity $\zeta$, and let $\omega = \zeta^6$. 
Let $S$ and $U$ be the matrices ${\rm diag}(1,\omega, \omega ^2)$ and ${\rm diag}(\zeta, \zeta, \zeta \omega)$, respectively, and set
\begin{align*}
T=     \left( \begin{array}{ccc}
     & 1 &  \\
     &  & 1 \\
1     &  & 
\end{array} \right),\ 
V=(\omega-\omega ^2)^{-1}      \left( \begin{array}{ccc}
1     & 1 & 1 \\
1     & \omega & \omega ^2 \\
1     & \omega ^2 & \omega
     \end{array} \right).
\end{align*}
Define $G_{72}$ to be the group  (of order 72) that is isomorphic to the projective image of $\langle S,T,UVU ^{-1},V\rangle$ in ${\rm PGL}_3(\C)$ and $G_{216}$ to be the group  (of order 216) that is isomorphic to the projective image of $\langle S,T,U,V\rangle$.

From Martin (Section 8.2 of~\cite{Ma04}) we have that if an irreducible complex 3-dimensional group representation $\rho$ has projective image isomorphic to $G_{72}$, $G_{216}$, $A_6$, or ${\rm PSL}_2 (\F_7)$, then ${\rm Ad} (\rho)$ is irreducible. 
For any finite group $G$ there exist number fields $L$ and $K$ such that ${\rm Gal}(L / K)\simeq G$, and so one can construct a representation $\tau(\rho)$ of ${\rm Gal}(\bar{K} / K)$ corresponding to $\rho$. The strong Artin conjecture would imply that, in each of these four cases, $\tau(\rho)$ corresponds to a cuspidal automorphic representation for ${\rm GL}(3)/K$ whose adjoint lift is cuspidal.

With regard to representations of higher dimension, a result from~\cite{Ch15} says that for any odd prime $p$ and positive integer $k$, there exists a group $G$ of order $p ^{2k + 1}$ and a complex irreducible representation $\rho$ of $G$ of degree $p ^k$ such that ${\rm Ad} (\rho)$ is an irreducible representation of $G$. Again applying the strong Artin conjecture provides conjectural examples of cuspidal automorphic representations for ${\rm GL}(p^k)$ whose adjoint lift would be cuspidal.

\section{Proof of Theorems~\ref{t1} and~\ref{t2}}\label{pf1}
\subsection{Proof of Theorem~\ref{t1}}  \label{st1}
For the whole of this section, we assume that the adjoint lifts of  unitary cuspidal automorphic representations $\pi$ and $\pi'$ for ${\rm GL}(n) /F$ are automorphic.

Given the Clebsch--Gordon decompositions of tensor powers of representations, for $X=X(\pi)$ we have the identity
\begin{align*}
L^X(s, {\rm Ad}\pi \times {\rm Ad} \pi) L^X(s, {\rm Ad}\pi)^2 \zeta^X(s) = L^X(s, \pi \times \widetilde{\pi}\times \pi \times \widetilde{\pi}),
\end{align*}
If we furthermore assume that the adjoint lifts of $\pi$ and $\pi'$ are cuspidal, then the self-duality of the adjoint lift implies that $L^X(s, {\rm Ad}\pi \times {\rm Ad} \pi)$ has a simple pole at $s=1$. Since $L^X(s, {\rm Ad}\pi)$ is invertible at $s=1$, we know that the incomplete quadruple product $L$-function $L^X(s, \pi \times \widetilde{\pi}\times \pi \times \widetilde{\pi})$ has a pole of order two at $s=1$. The same holds for $L^X(s, \pi' \times \widetilde{\pi'}\times \pi' \times \widetilde{\pi'})$. \\

For $X = X (\pi,\pi')$, we have
\begin{align*}
L^X(s, {\rm Ad}\pi \times {\rm Ad} \pi' ) L^X(s, {\rm Ad}\pi) & L^X(s, {\rm Ad}\pi') \zeta^X(s) =L^X(s, \pi \times \widetilde{\pi}\times \pi' \times \widetilde{\pi'}).
\end{align*}
If the adjoint  lifts for $\pi$ and $\pi'$ are equal, we have that $L^X(s, {\rm Ad}\pi \times {\rm Ad} \pi')$ has a simple pole at $s=1$. Combined with the fact that both $L^X(s, {\rm Ad}\pi)$ and $L^X(s, {\rm Ad}\pi')$ are invertible at $s=1$, we would conclude that $L^X(s, \pi \times \widetilde{\pi}\times \pi' \times \widetilde{\pi'})$ has a pole of order two at $s=1$.
On the other hand, if  the adjoint lifts for $\pi$ and $\pi'$ are not equal, then $L^X(s, {\rm Ad}\pi \times {\rm Ad} \pi')$ is invertible at $s=1$, and so $L^X(s, \pi \times \widetilde{\pi}\times \pi' \times \widetilde{\pi'})$ has a pole of order one at $s=1$.

So $L^X(s, \pi \times \widetilde{\pi}\times \pi' \times \widetilde{\pi'})$ has a pole of order at most two at $s=1$.\\

We recall the bounds towards the Ramanujan conjecture for a cuspidal automorphic representation $\pi$ for ${\rm GL}(n) / F$ due to Luo-Rudnick-Sarnak~\cite{LRS99}: Every Satake parameter $\alpha_{v,j}(\pi)$ for $\pi$ at a finite place $v$ satisfies the bound $$|\alpha_{v,j}(\pi)| \leq {\rm N}v^{1/2 - (n ^2 + 1)^{-1}}.$$
Applying these bounds in conjunction with the positivity of certain series coefficients, we obtain the inequalities (recall our notation $\ell(s):= \log \left({1}/({s-1})\right)$)
\begin{align} \label{eqn1}
\sum_{v}\frac{|a_v(\pi)|^4}{{\rm N}v^s} \leq   2 \ell(s) + O\left(1\right), 
\end{align} 
and
\begin{align}\label{eqn2}
\sum_{v}\frac{|a_v(\pi)|^2 |a_v(\pi')|^2}{{\rm N}v^s} \leq   2 \ell(s) + O\left(1\right),
\end{align}
as $s \rightarrow 1^+$.

Let $c(v)$, for $v \not \in S$, be an indicator function that takes the value 1 when $a_v(\pi) \neq a_v(\pi')$, and 0 otherwise.
Consider the following inequality, obtained via multiple applications of Cauchy-Schwarz and use of the identity $\bar{a}^2 b ^2 + a ^2 \bar{b}^2 \leq 2 |a|^2 |b|^2$, 
\begin{align}\notag 
&\sum_{v \not \in X}\frac{|a_v(\pi) - a_v(\pi')|^2 c(v)}{{\rm N}v^s} \leq
 \left(\sum_{v \not \in X}\frac{c(v)^2}{{\rm N}v^s}\right)^{1/2}  \cdot \\ \label{eqn4}
&\left[  \left(\sum_{v \not \in X}\frac{|a_v(\pi)|^4}{{\rm N}v^s}\right)^{1/2}  + \left(\sum_{v \not \in X}\frac{|a_v(\pi')|^4}{{\rm N}v^s}\right)^{1/2}  +  \left(\sum_{v \not \in X} 4\frac{|a_v(\pi)|^2 |a_v(\pi')|^2}{{\rm N}v^s}\right)^{1/2}   \right].
\end{align}
Given real-valued non-negative functions $f(x),g(x)$ and a real number $t$, we have the following identities
\begin{align*}
\lim_{x \rightarrow t^+} {\rm inf} (f(x) \cdot g(x)) &\leq \lim_{x \rightarrow t^+} {\rm sup} (f(x)) \cdot
\lim_{x \rightarrow t^+} {\rm inf} (g(x)),\\
\lim_{x \rightarrow t^+} {\rm sup} (f(x) + g(x)) &\leq \lim_{x \rightarrow t^+} {\rm sup} (f(x)) +
\lim_{x \rightarrow t^+} {\rm sup} (g(x)).
\end{align*}
We divide inequality~(\ref{eqn4}) by $\log \left(1 / (s-1)\right)$,
use the identities above for $t = 1$, and apply equations~(\ref{eqn1}) and~(\ref{eqn2}) to get
\begin{align*}
2 \leq (\sqrt{2} + \sqrt{2} + \sqrt{8}) \cdot \underline{\delta} (S)^{1/2},
\end{align*}
where we recall that $\underline{\delta} (S)= \underline{\delta} (S(\pi, \pi'))$ is the lower Dirichlet density of the set of places at which $a_v(\pi) \neq a_v(\pi')$. We obtain
\begin{align*}
\frac 18  \leq \underline{\delta} (S) ,
\end{align*}
proving part (a) of the theorem.\\

For part (b) of the theorem, since we assume that the adjoint lifts are distinct,  $L^X(s, {\rm Ad}\pi \times {\rm Ad}\pi')$ is invertible at $s=1$. We adjust the proof of part (a) accordingly to then get 
\begin{align*}
\frac{1}{3 + 2 \sqrt{2}} \leq \underline{\delta}(S),
\end{align*}
which finishes the proof of Theorem~\ref{t1}.

\subsection{Proof of Theorem~\ref{t2}}
We only provide a sketch of the proof of Theorem~\ref{t2}, as it follows the structure of~\cite{Wa14} (which only concerned the {\rm GL}(2) case), but with the difference that we can no longer rely on the automorphy (and cuspidality criteria) of the symmetric square, cube, and quartic powers (which are known for {\rm GL}(2)), and instead make use of our assumptions about certain adjoint lifts and  Rankin--Selberg products. In particular we assume, as indicated in the conditions of Theorem~\ref{t2}, that the adjoint lifts of $\pi$ and $\pi'$ are cuspidal and distinct, and that the Rankin--Selberg products $\pi \boxtimes \pi'$ and $\pi \boxtimes \widetilde{\pi'}$ are automorphic and cuspidal.

Given the identity, for $X=X(\pi,\pi')$,
\begin{align*}
L^X(s, \pi \boxtimes \widetilde{\pi} \times \pi \boxtimes \widetilde{\pi'}) = L^X(s, {\rm Ad}\pi \times \pi \boxtimes \widetilde{\pi'})L^X(s, \pi \boxtimes \widetilde{\pi'}),
\end{align*}
we note that the second $L$-function is invertible at $s=1$ (since it is of the form ${\rm GL}(n^2-1) \times {\rm GL}(n^2)$ where both components have been assumed to be cuspidal) and the same holds for the third $L$-function. Therefore the $L$-function on the left-hand side is also invertible at $s=1$.  The Luo--Rudnick--Sarnak bounds towards the Ramanujan conjecture then imply 
\begin{align*}
\sum_{v \not \in X}\frac{a_v(\pi)^2 \overline{a_v(\pi)} \overline{a_v(\pi')}}{{\rm N}v^s} = O\left(1\right) 
\end{align*}
as $s \rightarrow 1^+$. Similarly, the $L$-functions $L^X(s, \pi \boxtimes \widetilde{\pi} \times \widetilde{\pi} \boxtimes \pi'), L^X(s, \pi \boxtimes \pi' \times \widetilde{\pi'} \boxtimes \widetilde{\pi'})$ and $L^X(s, \widetilde{\pi} \boxtimes \pi' \times \pi' \boxtimes \widetilde{\pi'})$ are all invertible at $s=1$, and we draw the analogous conclusions about the associated Dirichlet series.

The $L$-function $L^X(s, \widetilde{\pi} \boxtimes \pi' \times \widetilde{\pi} \boxtimes \pi')$ has a simple pole at $s=1$ if $\widetilde{\pi} \boxtimes \pi'$ is self-dual, otherwise it is invertible there. We conclude that 
\begin{align*}
\sum_{v \not \in X}\frac{a_v(\pi)^2  \overline{a_v(\pi')}^2}{{\rm N}v^s}   =   \left\{ \begin{array}{cll}
\ell(s) + O\left(1\right) & \text{if }\widetilde{\pi} \boxtimes \pi' \text{ is self-dual},& \\
O\left(1\right) & \text{otherwise},&\\
     \end{array} \right.
\end{align*}
as $s \rightarrow 1^+$. 

We apply the results of this subsection, along with equations~(\ref{eqn1}) and~(\ref{eqn2}), to 
\begin{align}\label{cseq}
\sum_{}|a_v(\pi) - a_v(\pi')|^2 c(v) {\rm N}v^{-s} \leq
(\sum_{}c(v){\rm N}v^{-s})^{1/2} \cdot
(\sum_{}|a_v(\pi) - a_v(\pi')|^4{\rm N}v^{-s})^{1/2} 
\end{align}
where, as before, $c(v)$ is equal to one if $a_v(\pi) \neq a_v(\pi')$, and zero otherwise. We then obtain $$\underline{\delta}(S) \geq 2/5,$$  
proving Theorem~\ref{t2}. 

\section{The ${\rm GL}(3)$ case}\label{pf2}

We begin by collecting in one place all the bounds applicable to the ${\rm GL}(3)$ setting, arising from Theorems~\ref{t1},~\ref{t2}, and~\ref{gl3}. 
\begin{theorem}
Let $\Pi$ and $\Pi'$ be distinct unitary cuspidal automorphic representations for ${\rm GL}(3)$ over a number field $F$.\\
(a) If the adjoint lifts of $\Pi$ and $\Pi'$ are weakly automorphic,
\begin{align*}
\underline{\delta}(\Pi,\Pi') \geq 1/ 28
\end{align*}
(b) If $\Pi$ and $\Pi'$ are essentially self-dual, 
\begin{align*}
\underline{\delta}(\Pi,\Pi') \geq  2/(17+3\sqrt{21})   
\end{align*}
(c) If $\Pi$ and $\Pi'$ are essentially self-dual and not associated to a cuspidal automorphic representation for ${\rm GL}(2)$ of solvable polyhedral type, 
\begin{align*}
\underline{\delta}(\Pi,\Pi') \geq 1/12
\end{align*}
(i) If ${\rm Ad}(\Pi)$ and ${\rm Ad}(\Pi')$ are both automorphic cuspidal,  
\begin{align*}
\underline{\delta} (\Pi,\Pi') \geq {1}/{8}
\end{align*}
(ii) If ${\rm Ad}(\Pi)$ and ${\rm Ad}(\Pi')$ are both automorphic cuspidal, and furthermore distinct,
\begin{align*}
\underline{\delta}(\Pi,\Pi') \geq 1/(3 + 2 \sqrt{2}) 
\end{align*}
(iii) If the Rankin--Selberg products $\Pi \boxtimes \Pi'$ and $\Pi \boxtimes \widetilde{\Pi'}$, are automorphic cuspidal, and the adjoint lifts of $\Pi$, $\Pi'$ are automorphic cuspidal and distinct, 
\begin{align*}
\underline{\delta}(\Pi,\Pi') \geq {2}/{5}.
\end{align*}
\end{theorem}

Parts (i)-(iii) of this theorem arise from Theorems~\ref{t1} and~\ref{t2}. In this section we will prove parts (a)-(c), which is the content of Theorem~\ref{gl3}. 

\subsection{Proof of part (a)} 

If $\Pi$ has a weakly automorphic adjoint lift, denote the corresponding automorphic representation as $\tau$, which, for some finite set of places $T$, satisfies $L^T(s, \Pi, {\rm Ad})= L^T(s, \tau)$. For the adjoint lift of $\Pi'$, define $\tau'$ and $T'$ analogously. Let $R$ denote the (finite) set of places at which any of the automorphic representations that arise in this section are ramified, and let ${\rm arch}(F)$ be the set of archimedean places of $F$. Define 
$X := T \cup T' \cup R \cup {\rm arch}(F)$.

\begin{lemma}\label{L1}
Let $\Pi$ be a unitary cuspidal automorphic representation for ${\rm GL}(n) /F$ and let  $\chi$ be a non-trivial Hecke character over $F$. Assume that the adjoint lift of $\Pi$ is weakly automorphic and corresponds to the isobaric automorphic representation $\tau$. Then $\chi$ can only occur as a summand of $\tau$ at most once.
\end{lemma}

\begin{proof} Note that
\begin{align*}
L^X(s, \tau \otimes \overline{\chi})L^X(s, \overline{\chi})=
L^X(s, \Pi \times (\overline{\Pi} \otimes \overline{\chi}))
\end{align*}
The right-hand side $L$-function has a simple pole at $s=1$ if and only if $\overline{\Pi}$ admits a self-twist by $\overline{\chi}$. Otherwise, this $L$-function is invertible at $s=1$. This means that $\tau$ can only have $\chi$ as a summand at most once.
\end{proof}

\begin{lemma}
Let $\Pi$ be a unitary cuspidal automorphic representation for ${\rm GL}(3)/F$ with a weakly automorphic adjoint lift that corresponds to $\tau$. Let $\pi$ be a cuspidal automorphic representation for ${\rm GL}(2)/F$. Then $\pi$ can only occur as a summand for $\tau$ at most once.
\end{lemma}

\begin{proof} We have 
\begin{align*}
L^X(s, \tau \otimes \overline{\pi})L^X(s, \overline{\pi})=
L^X(s, \Pi \times (\overline{\Pi} \boxtimes \overline{\pi})).
\end{align*}
 If $\pi$ occurs as a summand for $\tau$ twice, then the first $L$-function in the equation has a pole of order two. Given the $L$-function on the right-hand side, this means that $\Pi \boxtimes \pi = \Pi \boxplus \Pi$.

From the proof of Theorem 8.1 of Ramakrishnan--Wang~\cite{RW04}, we know that if the automorphic product of a cuspidal automorphic representation $\Pi$ for {\rm GL}(3) with a cuspidal automorphic representation $\pi$ for {\rm GL}(2) has an isobaric decomposition into two cuspidal automorphic representations for {\rm GL}(3), then it must have the decomposition $\Pi \boxtimes \pi = (\Pi \otimes \nu)\boxplus (\Pi \otimes \nu \delta)$, where $\nu$ and $\delta$ are Hecke characters. In particular, $\delta$ is a (non-trivial) quadratic character, proving the lemma.
\end{proof}

\begin{remark}
One might ask if there is also an analogous lemma for {\rm GL}(3)-summands, namely, whether the isobaric automorphic representation corresponding to the weakly automorphic adjoint lift of a cuspidal automorphic representation for ${\rm GL}(3)$ cannot have the same {\rm GL}(3)-summand twice. This does not hold as there is the following set of counterexamples: Let $\pi$ be a cuspidal automorphic representation for ${\rm GL}(2)/F$ of tetrahedral type (which is equivalent to the condition that ${\rm Sym}^2 \pi$ is cuspidal but ${\rm Sym}^3 \pi$ is not). Then the adjoint lift $\Pi:={\rm Ad}\pi$ is a cuspidal automorphic representation for {\rm GL}(3). Using Clebsch--Gordon decompositions and the work of Kim--Shahidi~\cite{KS02} on functorial lifts, we get that 
\begin{align*}
L^X(s, \Pi, {\rm Ad}) = L^X(s, \mu)L^X(s, \mu ^2)L^X(s, {\rm Ad}\pi)^2 .\\
\end{align*}
\end{remark}

We now proceed similarly to the proof of Theorem~\ref{t1}, but this time we need to consider the various possibilities of how the isobaric automorphic representations $\tau, \tau'$ (corresponding to the weakly automorphic adjoint lifts of $\Pi$ and $\Pi'$, respectively) might decompose. We adjust our approach depending on whether the isobaric decomposition includes any cuspidal automorphic representations for ${\rm GL}(n)$, where $n \geq 3$.

If that is the case, then the weakest lower bound on $\underline{\delta}\left(\Pi, \Pi'\right)$ occurs if $\tau$ and $\tau'$ both have the isobaric decomposition $ \kappa \boxplus \kappa \boxplus \chi \boxplus \chi ^{-1} $, where $\kappa$ is a self-dual cuspidal automorphic representation for ${\rm GL}(3)$ and $\chi$ is a Hecke character of order at least three. Then
\begin{align*}
\sum_{v \not \in X} \frac{|a_v(\Pi^{(\prime)})|^4}{{\rm N}v^s} \leq 7 \ell(s) + O(1),\\
\sum_{v \not \in X} \frac{|a_v(\Pi)|^2|a_v(\Pi')|^2}{{\rm N}v^s} \leq 7 \ell(s) + O(1),
\end{align*}
which, using equation~(\ref{eqn4}), results in a bound of
\begin{align*}
\frac{4}{(4 \cdot \sqrt{7})^2} = \frac{1}{28} \leq \underline{\delta}(\Pi, \Pi').
\end{align*}

If there is no summand that is a cuspidal automorphic representation for {\rm GL}(n), $n \geq 3$, then the weakest lower bound on $\underline{\delta}(\Pi, \Pi')$ arises if $\tau$ and $ \tau'$ have the same isobaric decomposition into eight Hecke characters (all of which must be distinct, by Lemma~\ref{L1}).
In this case, 
$L^X(s, \tau \times \tau)$
has a pole at $s=1$ of order at most eight, and therefore, as $s \rightarrow 1^+$,
\begin{align*}
\sum_{v \not \in X} \frac{|a_v(\Pi)|^4}{{\rm N}v^s} \leq 9 \ell(s) + O(1),
\end{align*}
and similarly for $\Pi'$. We also have
\begin{align*}
\sum_{v \not \in X} \frac{|a_v(\Pi)|^2 |a_v (\Pi')|^2}{{\rm N}v^s} \leq 9 \ell(s) + O(1).
\end{align*}
Similar inequalities are also obtained for other Dirichlet series of the form 
\begin{align*}
\sum_{v \not \in X}\frac{a_v(\Pi)^w \overline{a_v(\Pi)^x} a_v(\Pi')^y \overline{a_v(\Pi')^z }}{{\rm N}v^s},
\end{align*}
where $w,x,y,z$ are non-negative integers that sum up to 4. Here, we rely in part on knowing the automorphy of the ${\rm GL}(2) \times {\rm GL}(3)$ Rankin--Selberg product, due to Kim--Shahidi~\cite{KS00}, and the associated cuspidality criteria, due to Ramakrishnan-Wang~\cite{RW04}.

Combining these Dirichlet series inequalities in conjunction with equation (\ref{cseq}) in the same manner as in the previous section, we then get the bound
\begin{align*}
\frac{1}{18} \leq \underline{\delta}(\Pi, \Pi').
\end{align*}

Note that, in this setting, the bound obtained is equal to that of Ramakrishnan's conjecture for the ${\rm GL}(3)$ case.

\subsection{Proof of parts (b) and (c)} Since $\Pi$ and $\Pi'$ are essentially self-dual, we can write them as ${\rm Sym}^2 \pi \otimes \nu$ and ${\rm Sym}^2 \pi' \otimes \nu'$, where $\pi, \pi'$ are cuspidal automorphic representations for ${\rm GL}(2)/F$ that are non-dihedral, and $\nu, \nu'$ are Hecke characters. Clebsch--Gordon decompositions imply 
\begin{align*}
L^X(s, \Pi, {\rm Ad}) = L^X(s,{\rm Sym}^4 \pi \otimes \omega ^{-2} \boxplus {\rm Sym}^2 \pi \otimes \omega ^{-1} )
\end{align*}
where $\omega$ is the central character for $\pi$; note that the symmetric square and symmetric fourth power lifts are known to be automorphic due to the work of  Gelbart--Jacquet~\cite{GJ78} and Kim~\cite{Ki03}, respectively. The analogous equation also holds for $\Pi'$.

Assume that $\pi$, $\pi'$ are not of solvable polyhedral type. Then their symmetric square and fourth power lifts are cuspidal~\cite{KS02}. The  $L$-function $L^X(s, \Pi \times \Pi \times \overline{\Pi} \times \overline{\Pi})$ can be rewritten as
\begin{align*}
L^X(s, (1 \boxplus {\rm Sym}^4 \pi \otimes \omega ^{-2} \boxplus {\rm Sym}^2 \pi \otimes \omega ^{-1} ) \times (1 \boxplus {\rm Sym}^4 \pi \otimes \omega ^{-2} \boxplus {\rm Sym}^2 \pi \otimes \omega ^{-1} ))
\end{align*}
and so we see it has a pole of order 3 at $s=1$. Taking logarithms and using positivity we get 
\begin{align*}
\sum_{v \not \in X} \frac{|a_v(\Pi)|^4}{{\rm N}v^s} \leq 3 \ell(s) + O(1),
\end{align*}
as $s \rightarrow 1^+$, and similarly for $\Pi'$. The same approach also gives 
\begin{align*}
\sum_{v \not \in X} \frac{|a_v(\Pi)|^2 |a_v (\Pi')|^2}{{\rm N}v^s} \leq 3 \ell(s) + O(1).
\end{align*}
Using these results in conjunction with equation~(\ref{eqn4}), we obtain 
\begin{align*}
\underline{\delta }(\Pi,\Pi') \geq 1/12.
\end{align*}

If one or both of $\pi$ , $\pi'$ are of solvable polyhedral type, then the associated symmetric fourth power lifts are not cuspidal~\cite{KS02}: If $\pi$ is of tetrahedral type, then 
\begin{align*}
{\rm Sym}^4 \pi = ({\rm Sym}^2 \pi \otimes \omega ) \boxplus \omega ^2 \mu \boxplus \omega ^2 \mu ^2,
\end{align*}
where $\omega $ is the central character of $\pi$ and $\mu$ is a non-trivial (cubic) Hecke character such that ${\rm Ad}\pi = {\rm Ad}\pi \otimes \mu$. If $\pi$ is of octahedral type (meaning that its symmetric cube lift is cuspidal but its symmetric fourth power lift is not), then 
\begin{align*}
{\rm Sym}^4 \pi = \sigma \boxplus ({\rm Ad}\pi \otimes \eta),
\end{align*}
where $\sigma$ is a dihedral cuspidal automorphic representation for {\rm GL}(2) and $\eta$ is a non-trivial quadratic Hecke character.

Using these decompositions, we obtain
\begin{align*}
\sum_{v \not \in X} \frac{|a_v(\Pi)|^4}{{\rm N}v^s} &\leq 7 \ell(s) + O(1),\\
\sum_{v \not \in X} \frac{|a_v(\Pi)|^2 |a_v (\Pi')|^2}{{\rm N}v^s} &\leq 7 \ell(s) + O(1),
\end{align*}
as $s \rightarrow 1^+$.
If exactly one of $\pi$, $\pi'$ is not of solvable polyhedral type, we proceed in the same manner as before, to obtain
\begin{align*}
\underline{\delta }(\Pi,\Pi') \geq    2/(17+3\sqrt{21})   
\end{align*}
If both $\pi$ and $\pi'$ are of solvable polyhedral type, we can instead apply the asymptotic results above (along with similar results for other related Dirichlet series) with inequality (\ref{cseq}) to obtain
\begin{align*}
\underline{\delta }(\Pi,\Pi') \geq 1/14.\\
\end{align*}

\subsection*{Acknowledgements} The author would like to thank Jim Cogdell and Kimball Martin for their helpful feedback on an earlier draft of this paper.

\end{document}